\documentclass[12pt,leqno]{amsart}
\numberwithin{equation}{section}
\usepackage{amssymb}
\usepackage{amsmath}
\usepackage{latexsym}
\allowdisplaybreaks
\newcommand{\eps}{\varepsilon}
\def\R{\mathbb R}
\def\C{\mathbb C}

\def\CC{\widehat{\mathbb C}}
\def\N{\mathbb N}

\def\arg{\operatorname{arg}}
\def\conn{\operatorname{conn}}
\def\crit{\operatorname{crit}}
\def\length{\operatorname{length}}
\def\dist{\operatorname{dist}}

\def\modulus{\operatorname{mod}}

\def\diam{\operatorname{diam}}

\def\ann{\operatorname{ann}}

\newtheorem{lemma}{Lemma}[section]
\newtheorem{thm}{Theorem}[section]
\newtheorem*{thma}{Theorem A}

\theoremstyle{definition}

\theoremstyle{remark}

\newtheorem{remark}{Remark}[section]

\title
[Boundary of Fatou components] {On the uniform perfectness of the
boundary of multiply connected wandering domains}
\author{Walter Bergweiler}\thanks{The first author is supported
by a Chinese Academy of Sciences Visiting Professorship for Senior
International Scientists, Grant No.\ 2010 TIJ10, the Deutsche
Forschungsgemeinschaft, Be 1508/7-1,  the EU Research Training
Network CODY and the ESF Networking Programme HCAA.}
\address{Mathematisches Seminar,
Christian--Albrechts--Universit\"at zu Kiel,
Lude\-wig--Meyn--Str.~4,
D--24098 Kiel,
Germany}
\email{bergweiler@math.uni-kiel.de}
\author{Jian-Hua Zheng}\thanks{The second author is supported by 
Grant No. 10871108 of NSF of
China.}
\address{Department of Mathematical Sciences, Tsinghua University,
100084, Beijing, P.\ R.\ China}
\email{jzheng@math.tsinghua.edu.cn}
\date{\today}

\begin{document}
\begin{abstract}
We investigate in which cases the boundary of a multiply
connected wandering domain of an entire function is uniformly perfect.
We give a general criterion implying that it is not
uniformly perfect. This criterion applies in particular to
examples of multiply connected wandering domains given
by Baker.  We also provide examples of infinitely connected
wandering domains whose boundary is uniformly perfect.
\end{abstract}

\maketitle

\section{Introduction and results}

The \emph{Fatou set} $F(f)$ of an  entire function $f$ is the subset
of the complex plane $\C$ where the iterates $f^n$ of $f$
form a normal family. Its complement is called the \emph{Julia set}
and denoted by $J(f)$; see~\cite{Bergweiler93} for an
introduction to and discussion
of these sets for transcendental functions.

The connected components of $F(f)$ are called \emph{Fatou
components}. For a Fatou component $U_0$ and $k\in\N$ there exists a
Fatou component $U_k$ containing $f^k(U_0)$. A~Fatou component $U_0$
is called a \emph{wandering domain} if $U_j\neq U_k$ for
$j\neq k$. While a famous theorem of Sullivan~\cite{Sullivan85} says
that rational functions do not have wandering domains, an example of
a transcendental entire function with wandering domains had been
constructed already before Sullivan's work by Baker~\cite{Baker76}.
The wandering domains in Baker's example are multiply connected.
Baker~\cite{Baker75} actually proved that multiply connected Fatou
components of transcendental entire functions are always wandering.

Baker's example~\cite{Baker76} was the function
\begin{equation}\label{1a}
f(z)=Cz^2
\prod_{k=1}^\infty\left(1+\frac{z}{r_k}\right),
\end{equation}
where $C$ is a constant and $(r_k)$ satisfies the recurrence
relation
\begin{equation}\label{1a1}
r_{k+1}=Cr_k^2\prod_{j=1}^{k}\left(1+\frac{r_k}{r_j}\right),
\end{equation}
with $r_1>1$ and $C>0$ chosen such that $C\exp(2/r_1)<1/4$ and
$Cr_1>1$, for example $C=1/(4e)$ and $r_1>4e$. Then $r_{k+1}\geq 2
r_k$ for all $k\in\N$
so that the product in~\eqref{1a} converges.
Baker showed that
$$f\left(\ann\left(0;a_k^2,\sqrt{a_{k+1}}\right)\right)\subset
\ann\left(0;a_{k+1}^2,\sqrt{a_{k+2}}\right)$$
for large~$k$,
where $\ann(a;r,R)=\{z\in\C:r<|z-a|<R\}$ for $0<r<R$ and $a\in\C$.
This implies that
$$
\ann\left(0;a_k^2,\sqrt{a_{k+1}}\right)\subset U_k
$$
for some multiply connected Fatou component $U_k$. In fact, Baker
had constructed the example and verified the above properties much
earlier~\cite{Baker63}, but in that paper the question whether the
$U_k$ are all different had remained open. It was only
in~\cite{Baker76} that he could prove that all the $U_k$ are
different and thus wandering domains.

Many properties of this example are typical for functions with
multiply connected Fatou components. We collect some of these
properties in the following theorem. Here $n(\gamma,a)$ denotes the
winding number of a curve $\gamma$ with respect to the point~$a$.
\begin{thma}
Let $f$ be a transcendental entire function with a multiply
connected wandering domain $U_0$ and, for $k\in\N$, let $U_k$ be the
component of $F(f)$ containing $f^k(U)$. Then we have the following:
\begin{itemize}
\item[(i)] $f^k|_{U_0}\to\infty$ as $k\to\infty$;
\item[(ii)] each $U_k$ contains a closed curve $\gamma_k$ satisfying
$n(\gamma_k,0)\neq 0$  for large $k$; in fact, if $\gamma_0$ is a
Jordan curve in $U_0$ which is not null-homotopic,  then
$\gamma_k=f^k(\gamma_0)$ has this property;
\item[(iii)] all $U_k$ are  bounded;
\item[(iv)]
each $U_k$ is contained in a bounded component of the complement of
$U_{k+1}$ for large~$k$;
\item[(v)] there exists sequences
$(r_k)$ and $(R_k)$ tending to $\infty$ such that
$\ann(0;r_k,R_k)\subset U_k$ for large $k$ and
$\lim_{k\to\infty}R_k/r_k= \infty$.
\end{itemize}
\end{thma}
Here (i) is a simple observation apparently first made by
T\"opfer~\cite[p.~67]{Toepfer}, and (ii) and (iii) are due to
Baker~\cite[Theorem~3.1]{Baker84}. (He states only the first part
of~(ii), but his proof gives the second one.) Next, (iv) is an easy
consequence of~(ii). Finally, (v) was proved in~\cite{Zheng06}.

Baker~\cite[Theorem~2]{Baker88} modified his construction to show that
there exists an entire function $f$ of the form
\begin{equation}\label{1c}
f(z)=C^2
\prod_{k=1}^\infty\left(1+\frac{z}{r_k}\right)^2
\end{equation}
which has a multiply connected Fatou component of infinite
connectivity; cf.\ section~\ref{baker-ex}. In the opposite
direction, Kisaka and Shishikura~\cite{KS} constructed examples
where the connectivity is finite, thereby answering a question of
Baker. In fact, they showed that for every $N\in\N$ there exists an
entire function with an $N$-connected Fatou component.

Next we recall that a closed subset $K$ of $\C$ is called
\emph{uniformly perfect} if there exists $c>0$ such that
if $a\in K$ and $0<r<\diam(K)$, then  $\ann(a;cr,r)\cap K\neq
\emptyset$. An equivalent condition is that there exists
$C>0$ such that the modulus
$\modulus(A)$ of
any annulus $A$ separating
two components of $K$ satisfies $\modulus(A)\leq C$.
Here by an annulus we mean a doubly connected domain.
The concept of uniform perfectness was introduced by Beardon
and Pommerenke~\cite{Beardon78,Pommerenke79} and has found
many applications in complex analysis.

It was proved independently by Ma\~n\'e and da Rocha~\cite{Mane},
Hinkkanen~\cite{Hinkkanen} and Eremenko~\cite{Eremenko92} that Julia
sets of rational functions are uniformly perfect. On the other hand,
the sequence $(a_k)$ in Baker's example~\eqref{1a} satisfies
$\lim_{k\to\infty}\sqrt{a_{k+1}}/a_k^2=\infty$ and this implies that $J(f)$
is not uniformly perfect for this function~$f$. In fact, it follows
from Theorem~A, part~(v), that the Julia set of an entire function
with a multiply connected Fatou component is never  uniformly
perfect.

Here we study the question when the boundary of a multiply connected
Fatou component is uniformly perfect. Clearly, this is the case for
Fatou components of finite connectivity, so it suffices to consider
infinitely connected Fatou components.

For a domain $U\subset\C$  and $a\in \CC\setminus U$ we denote by
$C(a,U)$ the component of $\CC\setminus U$ that contains~$a$, and we
put $C(a,U)=\emptyset$ if $a\in U$. The union of $U$ and its bounded
complementary components is denoted by~$\widetilde{U}$. Thus
$\widetilde{U}=\C\setminus C(\infty,U)$. The set of critical points
of $f$ is denoted by $\crit(f)$.
\begin{thm}\label{thm1}
Let $f$ be an entire transcendental function with a multiply
connected Fatou component $U_0$ and put $U_k=f^k(U_0)$ for $k\in\N$.
Denote by $l_k$ the number of critical points $c$ of $f$ in
$\widetilde{U_k}$ for which $f(c)\notin C(0,U_{k+1})$,
by $m_k$ the number of zeros of $f$ in $C(0,U_k)$ and by
$n_k$ the number of zeros of $f$ in $\widetilde{U_k}$.
Suppose that
\begin{equation}\label{1e}
l_k<m_k
\end{equation}
for all large~$k$. Suppose also that there are infinitely many $k$
such that
$U_k\cap \crit(f)\neq\emptyset$ and such that  $U_{k+1}$ contains an
annulus $A_{k+1}$ which separates  $U_{k+1}\cap f(\crit(f))$
 from $C(0,U_{k+1})$
 and satisfies
\begin{equation}\label{1f}
\frac{\modulus(A_{k+1})}{n_k-m_k}\to\infty.
\end{equation}
Then $\partial U_0$ is not uniformly perfect.
\end{thm}
It is not difficult to see that $f$ has $n_k-m_k$ critical points in
$\widetilde{U_k}\setminus C(0,U_k)$. In particular,  we have
$n_k>m_k$ if $U_k$ contains a critical point. Thus the denominator
in~\eqref{1f} is non-zero.
Note that $l_k\leq n_k-m_k$ so that~\eqref{1e} is satisfied in
particular if
\begin{equation}\label{1g}
n_k< 2m_k.
\end{equation}
We note that $n_k-m_k$ and hence $l_k$ is bounded
in the examples~\eqref{1a} and~\eqref{1c}.
Thus~\eqref{1e} and~\eqref{1g} are satisfied there.

While the hypothesis of Theorem~\ref{thm1} seem somewhat
complicated, the result works well for specific examples.
We use Theorem~\ref{thm1} to show that in Baker's original example
 of a multiply connected wandering domain
(i.e., the example given by~\eqref{1a} and~\eqref{1a1}), the
connectivity is infinite and the boundary is not uniformly perfect.
The question whether the connectivity of this domain is finite or
infinite had been raised by Baker~\cite{Baker88} and by Kisaka and
Shishikura~\cite{KS}; the question was repeated
in~\cite[p.~2946]{BergweilerOW} and~\cite[p.~312]{Schleicher10}.

More generally, we give a fairly complete discussion of functions of
order $0$ where the moduli $r_k$ of the zeros satisfy a recursion
formula similar to~\eqref{1a1}, with initial values chosen such that
$r_{k+1}\geq 2r_k$ for large~$k$. Our result also shows that an
infinitely connected wandering domain may have a uniformly perfect
boundary.
\begin{thm}\label{thm4}
Let $(r_k)$ and $(P_k)$ be sequences of positive numbers
satisfying
$r_{k+1}\geq 2r_k$ for large $k$,
\begin{equation}\label{thm4a}
\lim_{k\to\infty}\sqrt[k]{P_k}=1
\end{equation}
and
$$
r_{k+1}= P_k r_k^N \prod_{j=1}^{k}\left(1+\frac{r_k}{r_j}\right),
$$
for some non-negative integer~$N$ and all~$k$.
Let $C\in\C\setminus \{0\}$,
let $(a_k)$ be a sequence of complex numbers satisfying $|a_k|=r_k$
and define the entire function $f$ by
\begin{equation}\label{thm4c}
f(z)=Cz^N
\prod_{k=1}^\infty\left(1-\frac{z}{a_k}\right).
\end{equation}
Then there exist $K\in \N$ and a sequence $(\varepsilon_k)$ of
positive real numbers tending to $0$ such that
\begin{equation}\label{thm4d}
f\left(\ann\left(0;(1+\eps_k)r_k,(1-\eps_k)r_{k+1}\right)\right)\subset
\ann\left(0;(1+\eps_{k+1})r_{k+1},(1-\eps_{k+1})r_{k+2}\right)
\end{equation}
for $k\geq K$.

We denote, for $k\geq K$,
the Fatou component containing
$\ann\left(0;(1+\eps_k)r_k,(1-\eps_k)r_{k+1}\right)$ by $U_k$. If
$$
\limsup_{k\to\infty}{kP_k} >\frac{|C|}{2e}
\quad\text{or}\quad
\liminf_{k\to\infty}{kP_k} <\frac{|C|}{2e}
$$
then $U_k$ is infinitely connected for all~$k$, and if
$$
\limsup_{k\to\infty}{kP_k} >\frac{|C|}{2e}
\quad\text{or}\quad
\liminf_{k\to\infty}{kP_k} =0,
$$
then $\partial U_k$ is not uniformly perfect for all~$k$.
If
\begin{equation}\label{thm4g}
\limsup_{k\to\infty}{kP_k} <\frac{|C|}{2e}
\quad\text{and}\quad
\liminf_{k\to\infty}{kP_k} >0,
\end{equation}
then $\partial U_k$ is uniformly perfect for all~$k$.
\end{thm}
Noting that $P_k=C$ and thus $\lim_{k\to\infty} k P_k=\infty$
in Baker's first example of a wandering domain,
we deduce that
this wandering domain is infinitely connected and that its boundary
is not uniformly perfect.

In principle a similar discussion could be done for functions of the
form
$$f(z)=Cz^N \prod_{k=1}^\infty\left(1-\frac{z}{a_k}\right)^2,$$
but here we will only prove that in Baker's example~\eqref{1c} of an
infinitely connected wandering domain the boundary is also not
uniformly perfect; see section~\ref{baker-ex}.

In the  functions considered in Theorem~\ref{thm4}, as well as in
the Baker's example~\eqref{1c}, we have $n_k-m_k\leq 2$ and thus
the condition~\eqref{1f} just says that $\modulus(A_{k+1})\to\infty$.
 In the following example
 we have $n_k-m_k\to\infty$.
 The example shows that~\eqref{1f} is best possible in
some sense.

\begin{thm}\label{thm2}
Let $q_0$ be an even integer and put $a_0=\exp(q_0/2)$.
Define sequences $(q_k)$ and $(a_k)$ recursively by
$$q_{k+1}=\frac32 q_k^2
\quad\text{ and }\quad
a_{k+1}=\exp q_k.$$
Then
$$f(z)=z^2\prod_{k=0}^\infty\left(1-\frac{z}{a_k}\right)^{q_k}$$
defines an entire function $f$ and
if $q_0$ is sufficiently large,
then $f$ has an infinitely connected Fatou component
$U_0$ whose boundary is uniformly perfect.

Moreover, \eqref{1e} is satisfied,
each $U_k$ contains exactly one critical
point $c_k$ and $U_{k+1}$ contains an annulus $A_{k+1}$
separating $f(c_k)$ from $C(0,U_{k+1})$
such that
\begin{equation}\label{1h}
\frac{\modulus(A_{k+1})}{n_k-m_k}\geq c
\end{equation}
for some positive constant $c$ and all~$k$.
\end{thm}
There is an interesting difference between the example given by
Theorem~\ref{thm2} and the examples obtained from Theorem~\ref{thm4}
by choosing $(P_k)$ such that~\eqref{thm4g} holds. Our proofs will
show that in the example given by Theorem~\ref{thm2} the
complementary components of the wandering domain cluster only at the
``outer boundary'' while in the examples obtained from
Theorem~\ref{thm4} they cluster at the ``inner boundary''. We
discuss this in more detail in a remark at the end of
section~\ref{proofthm2}.

\section{Preliminaries}
We denote the connectivity of a domain $G$ in $\C$ by $\conn(G)$;
that is, $\conn(G)$ is the number of connected components of
$\CC\setminus G$. The following result is known as the
Riemann-Hurwitz formula; see, e.g.,~\cite[p.~7]{Steinmetz}.
\begin{lemma}\label{lemma1}
Let $G$ and $H$ be domains in
$\C$ and let $f:G\to H$ be a proper holomorphic map of degree $d$
with $m$ critical points, counting multiplicity.
Then
$$\conn(G)-2=d(\conn(H)-2)+m.$$
\end{lemma}
Here it is understood that if one of the domains is of infinite
connectivity, then so is the other one.

It follows from Theorem~A, part (iii), that if $f$ is an entire
function with a multiply connected wandering domain $U_0$ and if
$U_k$ is the component of $F(f)$ containing $f^k(U)$, then $f:U_k\to
U_{k+1}$ is a proper map. In particular, $U_{k}=f^k(U_0)$.

The following consequence of the Riemann-Hurwitz formula
can be found in~\cite[Lem\-ma~6]{Baker88} and~\cite[Theorem~A]{KS}.
\begin{lemma}\label{lemma1a}
 Let $f$ be a transcendental entire function with a multiply
connected wandering domain $U_0$.
If $\bigcup_{k=0}^\infty f^k(U)$ contains infinitely many critical points,
then $U_0$ is infinitely connected.
\end{lemma}
It was shown in~\cite{BRS} that the converse
also holds: if $\bigcup_{k=0}^\infty U_k$ contains only
finitely many critical points, then
$U_0$ has finite connectivity.
We will not need
this result, but we mention that it shows that the hypothesis in
Theorem~\ref{thm1} that infinitely many $U_k$ contain a critical
point is automatically fulfilled if $U_0$ is infinitely connected.

\begin{lemma}\label{lemma2}
Let $G$ and $H$ be simply-connected domains and let $f:G\to H$
be a proper holomorphic map of degree~$d$.
Let $A$ be an annulus in $H$ which does not contain
critical values.
Denote by $p$ the number of critical points $c\in G$ for which
$f(c)\in\widetilde{A}$, counting multiplicities. If $d>2p$, then
$f^{-1}(A)$ has a component $B$ such that $f:B\to A$ is univalent.
In particular, $\modulus(B)=\modulus(A)$.
\end{lemma}
\begin{proof}
Let $C_1,\dots,C_k$ be the components of
$f^{-1}(\widetilde{A})$
and let $r_j$ be the number of critical points in~$C_j$.
Then $\sum_{j=1}^k r_j=p$.
Now
$f:C_j\to \widetilde{A}$ is a proper map and its  degree is
$r_j+1$ by the the Riemann-Hurwitz formula, since $C_j$
and $\widetilde{A}$ are simply connected.
Thus $d=\sum_{j=1}^k (r_j+1)= k+\sum_{j=1}^k r_j=k+p$.
By hypothesis, we have $d>2p$ and this implies that $k>p$.
Thus there exists $j$ such that $C_j$ contains no critical
point. Hence $f:C_j\to \widetilde{A}$ is univalent and
the conclusion follows.
\end{proof}
We denote the density of the hyperbolic metric in a hyperbolic
domain $U$ by $\varrho_U$ and the hyperbolic length of a curve
$\gamma$ in $U$ by $\length(\gamma,U)$. Thus
$\length(\gamma,U)=\int_\gamma \varrho_U(z)|dz|$. The following
result is well-known; see~\cite[Theorem~2.3]{Sugawa98}
or~\cite[Proposition~3]{Zheng01}.
\begin{lemma}\label{lemma3}
Let $U$ be a hyperbolic domain. Then $\partial U$ is
uniformly perfect if and only if there exists $\delta>0$
such that $\length(\gamma,U)\geq \delta$ for each curve $\gamma$
in $U$ which is not null-homotopic.
\end{lemma}
The next lemma is also standard, but for convenience we
include the proof. Here $n(\gamma,a)$ denotes the winding
number of a curve $\gamma$ with respect to a point~$a$.
\begin{lemma}\label{lemma4}
Let $0<r<R$ and let $\gamma$ be a curve in $\ann(0;r,R)$.
Then
$$\length(\gamma,\ann(0;r,R))\geq \frac{2\pi^2 |n(\gamma,0)|}{\log (R/r)}.$$
\end{lemma}
\begin{proof}
We may assume that $r=1$.
The density of the hyperbolic metric
in $\ann(0;1,R)$ is given by
(see, e.g.,~\cite[p.~12]{McM94})
\[
\varrho_{\ann(0;1,R)}(z)=
\frac{\pi}{|z| \sin(\pi\log |z|/\log R) \log R}.
\]
In particular, we have
\[
\varrho_{\ann(0;1,R)}(z)\geq
\frac{\pi}{|z| \log R}
\]
and thus
$$
\begin{aligned}
\length(\gamma,\ann(0;1,R))
&=\int_\gamma\varrho_{\ann(0;1,R)}(z)|dz|
\geq  \frac{\pi}{\log R}\int_\gamma \frac{ |dz|}{|z|} \\
&\geq \frac{\pi}{\log R}\int_\gamma  |d \arg z|
\geq \frac{\pi}{\log R}\left|\int_\gamma  d \arg z\right|
=\frac{2\pi^2}{\log R}|n(\gamma,0)| .
\end{aligned}
$$
\end{proof}
The following result can be  found in~\cite[Theorem~3]{Zheng00}.
\begin{lemma}\label{lemma5}
Let $U,V$ be domains in $\C$ and let $f:U\to V$ be a proper
holomorphic map. Then $\partial U$ is uniformly perfect if and only
if $\partial V$ is uniformly perfect.
\end{lemma}
Lemma~\ref{lemma5}
implies that if an entire function $f$ has a multiply connected
wandering domain $U_0$ and if $U_k=f^k(U)$ is as before, then
$\partial U_0$ is uniformly perfect if
and only if $\partial U_k$ is uniformly perfect.

\section{Proof of Theorems~\ref{thm1}}
It follows from Theorem~A, part (ii), that
\begin{equation}\label{1d}
0\notin U_k\quad\text{ and }\quad 0\in \widetilde{U_k}
\end{equation}
for large~$k$. In view of Lemma~\ref{lemma5} and the remark
following it, we may assume that~\eqref{1d} holds for all $k\geq 0$.

Let now $k$ be an index such that $U_k$ contains a critical point
$c_k$ and let $A_{k+1}$ be an annulus as given in the hypothesis.
Choosing $A_{k+1}$ slightly smaller if necessary, we may assume that
$f(c)\notin \overline{A_{k+1}}$ for every critical point $c\in U_k$
and that $\overline{A_{k+1}}\subset U_{k+1}$. Thus there exists an
annulus $B_{k+1}\subset U_{k+1}$ separating $A_{k+1}$ and
$C(\infty,U_{k+1})$ such that  $f(c_k)\in  B_{k+1}$. We may also
assume that $A_{k+1}$ and $B_{k+1}$ are bounded by smooth curves.

Let $V_k=f^{-1}(B_{k+1})\cap U_k$. By the Riemann-Hurwitz
formula, $V_k$ is at least triply connected.
Thus there exists a component $X_k$ of $\CC\setminus V_k$
satisfying $X_k\neq C(0,V_k)$ and $X_k\neq C(\infty,V_k)$.
Thus $0\notin X_k$ and $X_k\subset \widetilde{U_k}$.
Moreover, $X_k\cap U_k$ contains a component $Y_k$ of $f^{-1}(A_{k+1})$.
Since $f(c)\notin A_{k+1}$ for every critical point $c\in U_k$,
we find that $Y_k$ is also an annulus and $f:Y_k\to A_{k+1}$ is a covering.
Moreover, $f:\widetilde{Y_k}\to\widetilde{A_{k+1}}$ is a proper
map and its degree $d_k$ equals the number of zeros of
$f$ in $\widetilde{Y_k}$.
Since $\widetilde{Y_k}\subset X_k\subset \widetilde{U_k}
\setminus C(0,U_k)$ we have
 $d_k\leq n_k-m_k$.  On the other hand, $d_k$ is equal to the
degree of the covering $f:Y_k\to A_{k+1}$. We thus have
$$
\modulus(Y_k)=\frac{\modulus(A_{k+1})}{d_k} \geq
\frac{\modulus(A_{k+1})}{n_k-m_k}.
$$
 Hence $U_k$
contains an annulus $Y_k$ with $0\notin \widetilde{Y_k}$ and
$\modulus(Y_k)\to\infty$ as $k\to\infty$.

 Let
now $p_{k-1}$ be the number of critical points $c$ of $f$ in
$\widetilde{U_{k-1}}$ for which $f(c)\in \widetilde{Y_k}$. Since
$\widetilde{Y_k}\subset X_k$ and $X_k\neq C(0,U_k)$ we have
$p_{k-1}\leq l_{k-1}$. Without loss of generality we may assume
that~\eqref{1e} holds for all $k$ and thus $p_{k-1}<m_{k-1}$. Since
$n_{k-1}$ is the degree of the proper map $f:\widetilde{U_{k-1}} \to
\widetilde{U_k}$ we have $n_{k-1}\geq m_{k-1}+l_{k-1}> 2p_{k-1}$. We
deduce from Lemma~\ref{lemma2} that there exists an annulus
$Y_k^1\subset U_{k-1}$ which is mapped univalently onto $Y_k$
by~$f$. Applying Lemma~\ref{lemma2} again we find an annulus
$Y_k^2\subset U_{k-2}$ which is mapped univalently onto $Y_k$
by~$f^2$. Inductively we thus obtain a non-trivial
 annulus $Y_k^k\subset U_0$ which
is mapped univalently onto $Y_k$ by~$f^k$. Since
$\modulus(Y_k^k)=\modulus(Y_k)\to\infty$ as $k\to\infty$ by our
hypothesis~\eqref{1f}, we conclude that $\partial U_0$ is not
uniformly perfect.

\section{Proof of Theorem~\ref{thm4}}
Since $r_{k+1}\geq 2 r_k$ for large $k$ we easily see that the
product defining $f$ converges. In fact, we have not only
$r_{k+1}\geq 2 r_k$,  but we can deduce from~\eqref{thm4a} that
\begin{equation}\label{p4a}
\frac{r_{k+1}}{r_k} =
\frac{P_k}{P_{k-1}}\left(\frac{r_k}{r_{k-1}}\right)^k\geq
\frac{P_k}{P_{k-1}}3^{k-k_0}\geq 2^k
\end{equation}
for some $k_0$ and large $k\geq k_0$.
In particular,
$$
\lim_{k\to\infty}\frac{r_{k+1}}{r_k}=\infty.
$$
This implies that if $w\in\C\setminus\{0\}$, then
\begin{equation}\label{p4c}
\left|\prod_{j=1}^{k-1}\left(1-\frac{w a_k}{a_j}\right)\right|
\sim \left|w^{k-1}
\prod_{j=1}^{k-1}\frac{a_k}{a_j}\right|
=
|w|^{k-1}
\prod_{j=1}^{k-1}\frac{r_k}{r_j}
\sim |w|^{k-1}\prod_{j=1}^{k-1}
\left(1+\frac{r_k}{r_j}\right)
\end{equation}
and
\begin{equation}\label{p4d}
\left|\prod_{j=k+1}^{\infty}\left(1-\frac{w a_k}{a_j}\right)\right|
 \to 1
\end{equation}
as $k\to\infty$.  Given $\beta>\alpha>0$, we actually find
that~\eqref{p4c} and~\eqref{p4d} hold uniformly for $\alpha\leq |w|
\leq \beta$. It follows that
\begin{equation}\label{p4d1}
\begin{aligned}
|f(wa_k)|
&\sim
\left|C\left(wa_k\right)^N\right|
|w|^{k-1}\prod_{j=1}^{k-1}
\left(1+\frac{r_k}{r_j}\right) |w-1| \\
&=|C| |w|^{k-1+N}r_k^N \prod_{j=1}^{k-1}
\left(1+\frac{r_k}{r_j}\right)|w-1| \\
&=|C| \frac{|w|^{k-1+N}|w-1|}{2P_k}r_{k+1}
\end{aligned}
\end{equation}
as $k\to\infty$, uniformly for $\alpha\leq |w| \leq \beta$.

Fix $\eps\in(0,1/2]$. For $|z|=(1+\eps)r_k=(1+\eps)|a_k|$ we deduce
from~\eqref{p4d1} that
$$
|f(z)| \geq (1-o(1)) \frac{|C| (1+\varepsilon)^{k-1+N} \eps}{2P_k}
r_{k+1}
$$
and thus
\begin{equation}\label{p4e}
|f(z)|\geq 2 r_{k+1}
\quad\text{for}\quad
|z|=(1+\eps)r_k
\end{equation}
if $k$ is sufficiently large.
Similarly we find
that
\begin{equation}\label{p4f}
|f(z)|\leq \frac12 r_{k+1}
\quad\text{for}\quad
|z|=(1-\eps)r_k
\end{equation}
for large~$k$.
 Moreover, the above reasoning and~\eqref{p4a}
 show that if $|z|=(1-\eps)r_k$, then
$$
|f(z)| \geq (1-o(1)) \frac{|C| (1-\varepsilon)^{k-1+N} \eps}{2P_k}
r_{k+1}\geq (1-o(1)) \frac{|C| (1-\varepsilon)^{k-1+N} \eps
2^k}{2P_k} r_k.
$$
Thus
\begin{equation}\label{p4g}
|f(z)|\geq 2 r_{k}
\quad\text{for}\quad
|z|=(1-\eps)r_k
\end{equation}
for large~$k$.

Finally, replacing $k$ by $k+1$ in~\eqref{p4f} yields that
 $|f(z)|\leq r_{k+2}/2$ for $|z|=(1-\eps)r_{k+1}$.
 Noting that  $(1-\eps)r_{k+1}\geq (1+\eps)r_k$ we deduce from the maximum principle
  that this inequality also
holds for $|z|=(1+\eps)r_k$. Hence
\begin{equation}\label{p4h}
|f(z)|\leq \frac12 r_{k+2}
\quad\text{for}\quad
|z|=(1+\eps)r_k.
\end{equation}
Now~\eqref{thm4d} follows from~\eqref{p4e}--\eqref{p4h} and
the  maximum and minimum principle.

Next we show that for large $k$  the function $f$ has
exactly one critical point
in the closed annulus
$$B_k=\overline{\ann\left(0;\sqrt{r_kr_{k-1}},\sqrt{r_{k+1}r_k}\right)}$$
and that if we denote this critical point by $c_k$, then
\begin{equation}\label{p4i}
c_k=\left(1-\frac{1}{k+N+\delta_k}\right)a_k
\end{equation}
for some sequence $(\delta_k)$ tending to~$0$. In order to do this
we note that if $z\in B_k$, then
 \begin{equation}\label{rouch}
\begin{aligned}
&\quad \ \left|\frac{f'(z)}{f(z)}-\frac{k-1+N}{z}-\frac{1}{z-a_k}\right|
=
\left|\sum_{j=1}^{k-1}\left(\frac{1}{z-a_j}-\frac{1}{z}\right)
+ \sum_{j=k+1}^{\infty}\frac{1}{z-a_j}\right| \\
&\leq
\sum_{j=1}^{k-1} \frac{r_j}{|z|(|z|-r_j)}
+ \sum_{j=k+1}^{\infty}\frac{1}{r_j-|z|}
\leq
\sum_{j=1}^{k-1} \frac{2r_j}{|z|^2}
+ \sum_{j=k+1}^{\infty}\frac{2}{r_j}
\leq
\frac{4r_{k-1}}{|z|^2}+\frac{4}{r_{k+1}} \\
&\leq
\left(\sqrt{\frac{r_{k-1}}{r_k}}+\sqrt{\frac{r_{k}}{r_{k+1}}}\right)
\frac{1}{|z|}
=o\left(\frac{1}{|z|}\right)
\end{aligned}
\end{equation}
as $k\to\infty$.

Using Rouch\'e's theorem we deduce from~\eqref{rouch} that the
difference between the number of zeros and poles in $B_k$ is the
same for $f'/f$ and the function given by
$$
z\mapsto \frac{k-1+N}{z}+\frac{1}{z-a_k},
$$
 provided $k$ is sufficiently large. We
conclude that if $k$ is large, then $f'$ has exactly one zero $c_k
\in B_k$. Moreover, \eqref{rouch} implies that
$$
\frac{k-1+N}{c_k}+\frac{1}{c_k-a_k}=o\left(\frac{1}{|c_k|}\right)
$$
as $k\to\infty$.
This yields~\eqref{p4i}.

It follows from~\eqref{p4d1} and~\eqref{p4i}
that
\begin{equation}\label{p4j}
|f(c_k)|
\sim
|C|
\frac{
\left|1-\frac{1}{k+N+\delta_k}\right|^{k-1+N}
\left|\frac{1}{k+N+\delta_k}\right|}{2P_k} r_{k+1}
\sim \frac{|C|}{2ekP_k}r_{k+1}
\end{equation}
as $k\to\infty$.

Suppose now that $\limsup_{k\to\infty} kP_k>|C|/(2e)$.
Then there exists $\eps>0$ such that
$$|f(c_k)|<(1-\eps)r_{k+1}$$
 for infinitely many~$k$.
On the other hand, we have
$$|f(c_k)|>(1-\eps)\frac{|C|}{2ekP_k}r_{k+1}
\geq (1-\eps)\frac{|C|3^{k-k_0}}{2ekP_{k-1}}r_k\geq k r_k$$ for
large $k$ by~\eqref{thm4a} and~\eqref{p4a}. Thus $f(c_k)\in U_k$ for
infinitely many $k$ and Lemma~\ref{lemma1a} implies that all $U_k$
are infinitely connected.

Moreover, noting that  $n_{k-1}-m_{k-1}\leq 2$ we see that~\eqref{1f},
with $k$ replaced by $k-1$, holds for
$$A_k=
\ann\left(0;(1+\eps_k)r_k,k r_k\right) \subset U_k.$$
Clearly, \eqref{1g} and hence~\eqref{1e} are also satisfied.
Theorem~\ref{thm1}
yields that $\partial U_k$ is not uniformly perfect for $k\geq
K$.

Similarly we deduce from~\eqref{p4a} and~\eqref{p4j}
 that if $\liminf_{k\to\infty}
kP_k<|C|/(2e)$, then
\begin{equation} \label{p4x}
(1+\eps)r_{k+1}<|f(c_k)|<2^k r_{k+1} \leq \frac12 r_{k+2}
\end{equation}
and hence $f(c_k)\in U_{k+1}$ for infinitely many~$k$. Again we
deduce from Lemma~\ref{lemma1a} that all $U_k$ are infinitely
connected. Moreover, if $\liminf_{k\to\infty} kP_k=0$, then
$|f(c_k)|/r_{k+1}\to\infty$ and with
$$A_{k+1}=\ann\left(0;(1+\eps_{k+1})r_{k+1},|f(c_k)|
\right)$$
 we deduce from Theorem~\ref{thm1} that $\partial
U_k$ is not uniformly perfect for $k\geq K$.

Suppose now that~\eqref{thm4g} holds. We show first that if $k$ is
large, then
\begin{equation}\label{p4k}
\left\{z\in\C: \left(1-\frac{1}{k}\right)r_k \leq
|z|<(1-\eps_k)r_{k+1}, \ |z-a_k|\geq \frac{1}{k}r_k\right\}
 \subset
U_k,
\end{equation}
but
\begin{equation}\label{p4l}
U_k\cap \overline{D\left( a_k,\frac{\delta}{k}r_k\right)}=\emptyset
\end{equation}
and
\begin{equation}\label{p4m}
U_k\cap \left[ \frac{1}{2} a_k,
\left(1-\frac{\tau}{k}\right) a_k\right]=\emptyset
\end{equation}
for certain $\tau,\delta>0$. (Here
 $[u,v]=\{u+t(v-u):0\leq t\leq 1\}$ is
 the line segment connecting two points $u,v\in\C$.)

It follows from~\eqref{p4f} and~\eqref{p4g} that in order to
prove~\eqref{p4k} it suffices to show that there exists $\eps>0$
such that
\begin{equation}\label{p4n}
|f(z)|\geq (1+\eps) r_{k+1} \quad\text{for}\quad |z|=
 \left(1-\frac{1}{k}\right)r_k
\end{equation}
and
\begin{equation}\label{p4o}
|f(z)|\geq (1+\eps) r_{k+1} \quad\text{for}\quad |z-a_k|=
\frac{1}{k}r_k
\end{equation}
for large~$k$.
Now~\eqref{p4d1} yields
$$|f(z)|\geq
(1-o(1))|C|\frac{\left(1-\frac{1}{k}\right)^{k-1+N}}{2kP_k}r_{k+1}
=(1-o(1))\frac{|C|}{2ekP_k}r_{k+1}
$$
for $|z|=(1-1/k)r_k$. Since $\limsup_{k\to\infty} kP_k<|C|/(2e)$ by
our assumption,~\eqref{p4n} follows. Essentially the same argument
also yields~\eqref{p4o} and thus we obtain~\eqref{p4k}.

On the other hand, if $|z-a_k|=\delta r_k/k$, then $|z|\leq
(1+\delta/k)r_k$ and we find by similar estimates as before that
$$|f(z)|\leq
(1+o(1))\frac{|C|\left(1+\frac{\delta}{k}\right)^{k-1+N}\delta}{2kP_k}r_{k+1}
=(1+o(1))\frac{|C|e^\delta \delta}{2kP_k}r_{k+1}.
$$
Since we assumed that $\liminf_{k\to\infty} kP_k>0$ we see that if
$\delta$ is sufficiently small, then
\begin{equation}\label{8c}
|f(z)|\leq
\frac12 r_{k+1}
\end{equation}
for large $k$. Thus $f(z)\in \widetilde{U_k}$. Since
$\widetilde{U_k}\subset C(0,U_{k+1})$ and thus $\widetilde{U_k}\cap
U_{k+1}=\widetilde{U_k}\cap f(U_{k})=\emptyset$, this
yields~\eqref{p4l}.

Also, if $z=s a_k$ where $1/2 \leq s\leq (1-\tau/k)$, then
$$|f(z)|\sim
\frac{|C|s^{k-1+N}(1-s)}{2P_k}
$$
by~\eqref{p4d1}. For $\tau\geq 2$ the function $s\mapsto
s^{k-1+N}(1-s)$ is increasing in $[1/2,(1-\tau/k)]$ and thus we
obtain
$$|f(z)|\leq
(1+o(1))\frac{|C|\left(1-\frac{\tau}{k}\right)^{k-1+N}\tau}{2kP_k}r_{k+1} \sim
\frac{|C|e^{-\tau} \tau}{2kP_k}r_{k+1}.
$$
If $\tau$ is chosen sufficiently large, our assumption that
$\liminf_{k\to\infty} kP_k>0$ implies~\eqref{8c} for large $k$. The
same argument as above now yields~\eqref{p4m}.

We also note that by construction we have
\begin{equation}\label{p4p}
U_k\subset \ann(0;(1-\eps_{k-1})r_{k},(1+\eps_{k+1})r_{k+1}) \subset
\ann(0;1,2r_{k+1})
\end{equation}
for large $k$.
We may assume that~\eqref{p4k}, \eqref{p4l}, \eqref{p4m}
and~\eqref{p4p}
hold for $k\geq K$.

In order to prove that $\partial U_K$ is uniformly perfect we use
Lemma~\ref{lemma3}. Let $\sigma_K$ be a curve in $U_K$ which is not
null-homotopic. If $n(\sigma_K,0)\neq 0$, then
\begin{equation}\label{p4q}
\length(\sigma_K,U_K)\geq \length(\sigma_K,\ann(0;1,2r_{K+1})) \geq
\frac{2\pi^2 }{\log (2r_{K+1})}
\end{equation}
by~\eqref{p4p} and Lemma~\ref{lemma4}. Suppose that
$n(\sigma_K,0)=0$ and put $\sigma_k=f^{k-K}(\sigma_K)$ for $k> K$.
By Theorem~A, part (ii), we have $n(\sigma_k,0)\neq 0$ for large~$k$.
Thus there exists $k\geq K$ such that $n(\sigma_k,0)=0$ and
$n(\sigma_{k+1},0)\neq 0$. It follows that $n(\sigma_k,a)\neq 0$ for
some zero $a$ of~$f$, and~\eqref{p4k} implies that we actually have
$n(\sigma_k,a_k)\neq 0$.  Let
$$V_k= \C\setminus \left( \overline{D\left( a_k,\frac{\delta}{k}r_k\right)} \cup
\left[ \frac12 a_k, \left(1-\frac{\tau}{k}\right) a_k\right]\right).
$$
Since $U_k\subset V_k$ by~\eqref{p4l} and~\eqref{p4m}, we have
\begin{equation}\label{p4s}
\length(\sigma_K,U_K)\geq \length(\sigma_k,U_k)\geq
\length(\sigma_k,V_k).
\end{equation}
We put
$$
T_k(z)=k\frac{z-a_k}{a_k} \quad\text{and}\quad W_k=T_k(V_k)
=\C\setminus \left( \overline{D(0,\delta)}\cup
\left[-\frac12 k,-\tau\right]\right).
$$
We may assume that $K\geq 4\tau$.
With
$$
W= \C\setminus \left( \overline{D(0,\delta)}\cup
[-2\tau,-\tau]\right)
$$
we then have $W_k\subset W$ for $k\geq K$. Since $a_k/2\in C(0,U_k)$ we
$n(\sigma_k,a_k/2)=0$. On the other hand, $n(\sigma_k,a_k)\neq 0$
and thus  $\sigma_k$ separates $[a_k/2,(1-\tau/k)a_k]$ and
$\overline{D(a_k,\delta r_k/k)}$. Hence $T_k(\sigma_k)$ is a curve
in $W$ which separates $\overline{D(0,\delta)}$ and
$[-2\tau,-\tau]$. Thus
\begin{equation}\label{p4t}
\length(\sigma_k,V_k)= \length(T_k(\sigma_k),W_k)\geq
\length(T_k(\sigma_k),W) \geq c
\end{equation}
for some positive constant $c$. Combining~\eqref{p4s}
and~\eqref{p4t} we obtain
\begin{equation}\label{p4u}
\length(\sigma_K,U_K)\geq c.
\end{equation}
Now~\eqref{p4q} and~\eqref{p4u} imply together with
Lemma~\ref{lemma4} that $\partial U_K$ is uniformly perfect.

\begin{remark}\label{remks}
As mentioned in the introduction,
Kisaka and Shishikura~\cite{KS} constructed an example of
an entire function function $f$ with a doubly connected
wandering domain.
In order to ensure that the wandering domains
do not contain critical points, which has to be avoided
by Lemma~\ref{lemma1a}, they construct it in such a way
that $f(0)=0$ and $f^2(c)=0$ for each critical point
$c$ of~$f$.
Their construction uses quasiconformal surgery,
but it turns out that the function obtained
is of order zero and can be written in
the form~\eqref{thm4c} with a sequence $(a_k)$ which
tends to $\infty$ rapidly.
The sequence $(c_k)$ of critical points
again satisfies~\eqref{p4i} and
the construction is such that
 $f(c_k)=a_{k+1}$ and thus  $f^2(c_k)=0$.
It follows from the arguments in the above proof
that
$$
a_{k+1}=f(c_k)\sim a_k^N \frac{C}{ek}\prod_{j=0}^{k-1}\frac{a_k}{a_j}
$$
as $k\to\infty$.
\end{remark}

\section{Proof of Theorem~\ref{thm2}}\label{proofthm2}
It follows easily from Rolle's theorem that
for a real polynomial with real zeros, each open interval
on the real axis bounded by two adjacent zeros contains
exactly one critical point, this critical point is simple, and
there are no further critical points except for multiple zeros.
Since $f$ is a locally uniform limit of polynomials with
real zeros, the above result also holds for~$f$.
We thus find that for $k\geq 0$ there
exists a critical point $c_k\in (a_{k},a_{k+1})$ and
except $0$ and one further critical point in the interval $(0,a_0)$
there are no critical points other than the $a_k$
and~$c_k$.

We will show that if $q_0$ is sufficiently large, then
the following properties are satisfied for all $k\geq 0$:
\begin{equation}\label{3a}
f\left(\ann\left(4a_k,a_{k+1}/4\right)\right)
\subset \ann\left(4a_{k+1},a_{k+2}/4\right),
\end{equation}
\begin{equation}\label{3b}
f\left(\ann\left(4a_k,\sqrt{a_{k+1}}\right)\right)
\subset \ann\left(4a_{k+1},\sqrt{a_{k+2}}\right),
\end{equation}
\begin{equation}\label{3c}
\frac{q_{k}}{2q_{k+1}}a_{k+1}<c_k<
\frac{2q_{k}}{q_{k+1}}a_{k+1},
\end{equation}
and
\begin{equation}\label{3d}
\sqrt{a_{k+2}}<f(c_k)<\frac14 a_{k+2}.
\end{equation}
Suppose that \eqref{3a}--\eqref{3d} hold.
Then $f$ has a wandering $U_0$ containing
$\ann\left(4a_0,a_{1}/4\right)$.

We consider the annulus $X_k=D\left(0,\sqrt{a_{k+1}}\right)
\setminus C(0,U_k)$. If $k\geq 1$, then $X_k$ contains no critical
values by~\eqref{3d} and thus the components of $f^{-1}(X_k)$  are
also annuli. In particular, this holds for the component
 $Y_{k-1}$ of $f^{-1}(X_k)$
whose boundary intersects $C(0,U_{k-1})$. By~\eqref{3b} we have
$Y_{k-1}\supset X_{k-1}$ and thus $f(X_{k-1})\subset
f(Y_{k-1})=X_k$. It follows that $X_k\subset F(f)$ and thus
$X_k\subset U_k$ for all $k\geq 0$. This implies that
\begin{equation}\label{3d1}
a_j\in C(0,U_k) \quad\text{for } 0\leq j\leq k.
\end{equation}
We also note that
\begin{equation}\label{3d2}
U_k\subset \ann\left(a_k/4,4a_{k+1}\right)
\subset \ann\left(1,4a_{k+1}\right)
\end{equation}
for $k\geq 1$ and
\begin{equation}\label{3d3}
U_0\subset \ann\left(\delta,4a_{1}\right)
\end{equation}
for some $\delta>0$, since $0$ is a superattracting fixed point.

Similarly as in the proof of Theorem~\ref{thm4} we use
Lemma~\ref{lemma3} to show that $\partial U_0$ is uniformly perfect.
So let $\sigma_0$ be a Jordan curve in $U_0$ which is not
null-homotopic. If $n(\sigma_0,0)\neq 0$, then
\begin{equation}\label{3d4}
\length(\sigma_0,U_0)\geq \frac{2\pi^2}{\log (4a_1/\delta) }
\end{equation}
by~\eqref{3d3} and Lemma~\ref{lemma4}.

We now assume that $n(\sigma_0,0)=0$. Put $\sigma_k=f^k(\sigma_0)$.
By Theorem~A, part (ii), we have $n(\sigma_k,0)\neq 0$ for large $k$ and
thus there exists $k\geq 1$ such that $n(\sigma_k,0)\neq 0$ while
$n(\sigma_{k-1},0)= 0$. It follows that $n(\sigma_{k-1},a)\neq 0$
for some zero $a$ of~$f$. Using~\eqref{3d1} we see that
$n(\sigma_{k-1},a_k)\neq 0$ while $n(\sigma_{k-1},a_j)=0$ for $j\neq
k$. Since $a_k$ is a zero of multiplicity $q_k$ this implies that
$$n(\sigma_k,0)=n(f(\sigma_{k-1}),0)=q_k n(\sigma_{k-1},a_k).$$
In particular, $|n(\sigma_k,0)|\geq q_k$.
Combining this with~\eqref{3d2} and Lemma~\ref{lemma4} we obtain
$$
\length(\sigma_k,U_k)
\geq
\length(\sigma_k,\ann(1,4a_{k+1}))
\geq \frac{2\pi^2 q_k }{\log (4a_{k+1}) }
=\frac{2\pi^2 q_k }{q_k+\log 4}
\geq \pi^2.
$$
On the other hand, $\length(\sigma_k,U_k)\leq \length(\sigma_0,U_0)$
and thus we obtain
\begin{equation}\label{3d5}
\length(\sigma_0,U_0) \geq \pi^2
\end{equation}
in this case.
Thus for each Jordan curve $\sigma_0$ in $U_0$ which
is not null-homotopic we have~\eqref{3d4} or~\eqref{3d5}.
Lemma~\ref{lemma3} now implies that $\partial U_0$ is
uniformly perfect.

Next we note that $l_k=1$ since
$f(c_j)\in U_{j+1}\subset C(0,U_{k+1})$ for $j<k$ so that
$c_k$ is the only critical point of
$f$ in $\widetilde{U_k}$ which is not
mapped into $C(0,U_{k+1})$.
We also have
$$
n_k=2+\sum_{j=0}^{k+1}q_j
\quad\text{and}\quad m_k=2+\sum_{j=0}^k q_j
$$
by~\eqref{3d1}. Thus~\eqref{1e} is satisfied.

In order to prove~\eqref{1h} we note that by~\eqref{3d} the annulus
$$A_{k+1}=\ann\left(4a_{k+1},\sqrt{a_{k+2}}\right)$$
separates $f(c_k)$ and $C(0,U_{k+1})$.
Now
$$
\frac{\modulus(A_{k+1})}{n_k-m_k}=\frac{\log (\sqrt{a_{k+2}}/4a_{k+1})}{2\pi q_{k+1}}
=\frac{ q_{k+1}/2-\log 4-q_k}{2\pi q_{k+1}}\to\frac{1}{4\pi}
$$
as $k\to\infty$, and this proves~\eqref{1h}.

It remains to prove \eqref{3a}--\eqref{3d}.
In order to do this, we note first that
$q_{k+1}\geq 3q_0q_k/2$
and thus
$q_{k}\geq (3 q_0/ 2)^{k-j} q_j$
for $k>j$. Thus
$$
\sum_{j=0}^{k-1}q_j \leq
q_{k-1}\sum_{j=0}^{k-1} \left(\frac{2}{3q_0}\right)^{j}
$$
which implies that given $\varepsilon>0$ we can achieve
\begin{equation}\label{p3a}
\sum_{j=0}^{k-1}q_j \leq
(1+\varepsilon)q_{k-1}
\end{equation}
by choosing $q_0$ large.
In particular, $q_{k-2}\leq
\varepsilon q_{k-1}$.

We also note that the sequence $(a_k)$ tends to $\infty$ very
rapidly and using this it is not
difficult to see that we can achieve
\begin{equation}\label{8a}
\prod_{j=k+1}^{\infty}\left(1+\frac{4a_k}{a_j}\right)^{q_j}\leq 2
\quad\text{and}\quad
\prod_{j=k+1}^{\infty}\left(1-\frac{4a_k}{a_j}\right)^{q_j}\geq \frac12
\end{equation}
for all $k$ by choosing $q_0$ large. We can also  achieve
\begin{equation}\label{8b}
1+\frac{a_k}{4a_j}\leq a_k
\quad\text{and}\quad
\frac{a_k}{4a_j} -1\geq \frac{a_k}{5a_j}
\end{equation}
for $0\leq j<k$.
For $k\geq 1$ and $|z|=a_k/4$ we then have
$$
\begin{aligned}
|f(z)| & \leq  \frac{1}{16}a_k^2
\prod_{j=0}^{k-1}\left(1+\frac{a_k}{4a_j}\right)^{q_j}
\left(\frac54 \right)^{q_k}
\prod_{j=k+1}^{\infty}\left(1+\frac{a_k}{4a_j}\right)^{q_j}\\
& \leq \frac{1}{8}a_k^2 \prod_{j=0}^{k-1}a_k^{q_j}
\left(\frac54 \right)^{q_k}  \\
&= \frac{1}{8} \exp \left(2q_{k-1}+q_{k-1} \sum_{j=0}^{k-1}q_j
+q_k\log\frac54\right) \\
&\leq\frac{1}{8} \exp \left(2\eps q_k (1+\eps)q_{k-1}^2
+q_k\log\frac54\right) \\
&= \frac{1}{8} \exp
\left(\left(2\eps+(1+\eps)\frac23+\log\frac54\right)q_k\right)
\end{aligned}
$$
by~\eqref{p3a}, \eqref{8a} and~\eqref{8b}. Choosing  $\eps$ small, which we can achieve by
choosing $q_0$ large, we obtain
\begin{equation}\label{p3b}
|f(z)|\leq \frac18\exp q_k =\frac18 a_{k+1} \quad\text{for}\quad
|z|=\frac14 a_k.
\end{equation}

For $k\geq 1$ and $|z|=a_k/4$ we also have
$$
\begin{aligned}
|f(z)| & \geq  \frac{1}{16}a_k^2
\prod_{j=0}^{k-1}\left(\frac{a_k}{4a_j}-1\right)^{q_j} \left(\frac34
\right)^{q_k}
\prod_{j=k+1}^{\infty}\left(1-\frac{a_k}{4a_j}\right)^{q_j}\\
& \geq \frac{1}{32}a_k^2 \prod_{j=1}^{k-1}
\left(\frac{a_k}{5a_j}\right)^{q_j}
\left(\frac34 \right)^{q_k}  \\
&= \frac{a_k}{32} \exp \left(q_{k-1}+ \sum_{j=1}^{k-1}q_j
(q_{k-1}-q_{j-1}-\log 5) + q_k\log\frac34\right).
\end{aligned}
$$
by~\eqref{8a} and~\eqref{8b}.
Similarly as before we see that we also have
$$
q_{k-1}+ \sum_{j=1}^{k-1}q_j
(q_{k-1}-q_{j-1}-\log 5)
\geq q_k(1-\eps)\frac23
$$
for large $q_0$ and since we may
choose $\eps$ such that
$(1-\eps)2/3+\log(3/4)>0$,
we see that
\begin{equation}\label{p3c}
|f(z)|\geq 4 a_{k} \quad\text{for}\quad |z|=\frac14 a_k.
\end{equation}
For $k\geq 1$ and $|z|=4 a_k$ we have
$$|f(z)|\geq 16 a_k^2
\prod_{j=0}^{k-1}\left(\frac{4a_k}{a_j}-1\right)^{q_j} 3^{q_k}
\prod_{j=k+1}^{\infty}\left(1-\frac{4a_k}{a_j}\right)^{q_j}
$$
and thus
\begin{equation}\label{p3d}
|f(z)|\geq 8\cdot 3^{q_k}\geq  4 a_{k+1} \quad\text{for}\quad |z|=4
a_k
\end{equation}
by~\eqref{8a}.
If $|z|=4a_0$, then
$$|f(z)|\geq 16 a_0^2
 3^{q_0}
\prod_{j=1}^{\infty}\left(1-\frac{4a_0}{a_j}\right)^{q_j}\geq 8a_0^2
3^{q_0}= 8a_1 3^{q_0}\geq 4 a_1
$$
so that~\eqref{p3d} also holds for $k=0$. By~\eqref{p3b} we have
$|f(z)|\leq a_{k+2}/4$ for $|z|= a_{k+1}/4$. Since $a_{k+1}/4>4a_k$
for large $q_0$ we see that we also have
\begin{equation}\label{p3e}
|f(z)|\leq \frac14  a_{k+2} \quad\text{for}\quad |z|=4 a_k.
\end{equation}
Now~\eqref{3a} follows from~\eqref{p3b}--\eqref{p3e}.

For $k\geq 1$ and $|z|=\sqrt{a_k}$ we see similarly as before that
$$
\begin{aligned}
|f(z)|^2 & \leq a_k^2
\prod_{j=0}^{k-1}\left(1+\frac{\sqrt{a_k}}{a_j}\right)^{2q_j}
\left(1+\frac{1}{\sqrt{a_k}} \right)^{2q_k}  2\\
&\leq 2a_k^2 \prod_{j=0}^{k-1}a_k^{2q_j}
\left(1+\eps \right)^{2q_k} \\
&=\exp\left(\log 2+2q_{k-1} +q_{k-1}\sum_{j=0}^{k-1} q_j
+2q_k\log(1+\eps)\right) \\
&\leq \exp\left(\left( \frac23+\eps+2\log(1+\eps)\right)q_k\right) \\
&\leq a_{k+1}
\end{aligned}
$$
for large $q_0$ and small~$\eps$. This yields~\eqref{3b}.

The location of the critical points could  be determined by
Rouch\'e's theorem as in the proof of Theorem~\ref{thm4}.
Alternatively, to prove~\eqref{3c} it suffices that with
$$x_k=\frac{q_{k-1}}{2q_k}a_k
\quad\text{and}\quad
 y_k=\frac{2q_{k-1}}{q_k}a_k
 $$
 we have
 $$
 f'(x_k)>0
\quad\text{and}\quad f'(y_k)<0.
$$
Since $f(x)\geq 0$ for all $x\in\R$ it suffices to obtain these
inequalities with $f'$ replaced by $f'/f$. We have
$$
\frac{f'(z)}{f(z)} =\frac{2}{z} + \sum_{j=0}^\infty
\frac{q_j}{z-a_j}.
$$
Thus
$$
\frac{f'(x_k)}{f(x_k)} \geq \sum_{j=k-1}^\infty
\frac{q_j}{x_k-a_j}=\frac{q_{k-1}}{x_k-a_{k-1}}-\frac{q_k}{a_k-x_k}
 -\sum_{j=k+1}^\infty\frac{q_j}{a_j-x_k}.
$$
For large $q_0$ we have
$$
\frac{q_{k-1}}{x_k-a_{k-1}}
 =\frac{2q_k q_{k-1}}{q_{k-1}a_k-2q_k a_{k-1}}
\geq \frac74\frac{q_k}{a_k}
$$
while
$$
\frac{q_k}{a_k-x_k}
=\frac{2q_k^2}{2q_ka_k-q_{k-1}a_k}\leq\frac54\frac{q_k}{a_k}
$$
and
$$
\sum_{j=k+1}^\infty\frac{q_j}{a_j-x_k}\leq
 2\sum_{j=k+1}^\infty\frac{q_j}{a_j}\leq\frac14\frac{q_k}{a_k}.
$$
The last four inequalities yield that $f'(x_k)>0$. Similarly we find
for large $q_0$ that
$$
\begin{aligned}
\frac{f'(y_k)}{f(y_k)}
 &\leq \frac{2}{y_k}+ \sum_{j=0}^{k-1}
\frac{q_j}{y_k-a_j}-\frac{q_k}{a_k-y_k}\\
&=\frac{q_k}{q_{k-1}a_k}+ \sum_{j=0}^{k-1}
 \frac{q_jq_k}{2q_{k-1}a_k-q_ka_j}-\frac{q_k^2}{q_ka_k-2q_{k-1}a_k}\\
&\leq \frac{q_k}{q_{k-1}a_k}+\frac58
\sum_{j=0}^{k-1}\frac{q_jq_k}{q_{k-1}a_k}-\frac78\frac{q_k}{a_k} \\
&=\frac{q_k}{a_k} \left( \frac{1}{q_{k-1}} + \frac58
\frac{1}{q_{k-1}}\sum_{j=0}^{k-1}q_{k-1}-\frac78\right) \\
&<0
\end{aligned}
$$
and thus $f'(y_k)<0$. This completes the proof of~\eqref{3c}.

The right inequality of~\eqref{3d} follows from~\eqref{3a}
and~\eqref{3c}.
To prove the left one we note that
\begin{equation}\label{p3f}
\begin{aligned}
f(c_k)
&\geq
a_k^2 \prod_{j=0}^{k-1}\left(\frac{x_k}{a_j}-1\right)^{q_j}
\left(1-\frac{x_k}{a_k}\right)^{q_k}
\prod_{j=k+1}^\infty\left(1-\frac{x_k}{a_j}\right)^{q_j}
\\
&\geq
\left(\frac{x_k}{2a_{k-1}}\right)^{q_{k-1}}
\left(1-\frac{x_k}{a_k}\right)^{q_k}
\\
&=
\left(\frac{q_{k-1}a_k}{4q_ka_{k-1}}\right)^{q_{k-1}}
\left(1-\frac{2q_{k-1}}{q_k}\right)^{q_k}\\
&\geq
\left(\frac{q_{k-1}a_k}{4q_ka_{k-1}}\right)^{q_{k-1}}
\exp\left(-3q_{k-1}\right).
\end{aligned}
\end{equation}
Noting that
$$
q_{k}=\frac32 q_{k-1}^2 =\frac32 (\log a_k)^2
$$
and
$$
a_{k-1}=\exp q_{k-2} =\exp \sqrt{\frac23 q_{k-1}}
\leq \exp \sqrt{q_{k-1}}=\exp \sqrt{\log a_k}
$$
we see that for given $\delta>0$ we can achieve
$$
q_{k}\leq a_k^\delta
\quad\text{and}\quad
a_{k-1}\leq a_k^\delta
$$
by choosing $q_0$ large.
Thus, given $\eps>0$, we deduce from~\eqref{p3f}
that
$$f(c_k) \geq
a_{k}^{(1-\eps)q_{k-1}}\exp\left(-3q_{k-1}\right)
=\exp\left( (1-\eps)q_{k-1}^2-3q_{k-1}\right)
\geq \exp\left( (1-2\eps)\frac23 q_k\right)
$$
for large $q_0$. For small $\eps>0$ we thus have
$$
f(c_k)\geq \exp\left( \frac12 q_k\right)=\sqrt{a_{k+1}}.
$$
This completes the proof of~\eqref{3d} and thus the proof
of Theorem~\ref{thm2}.

\begin{remark}\label{cluster}
Let$f$ be an entire transcendental function with a
multiply connected wandering domain $U_0$ and put
$U_k=f^k(U_0)$ as before.
By Theorem~A, part~(iii), the $U_k$ are all bounded
and thus
$\partial U_k\cap C(\infty,U_k)$ is connected.
We call $\partial U_k\cap C(\infty,U_k)$ the \emph{outer boundary}
of $U_k$ and denote it by $\partial_\infty U_k$.
By Theorem~A, part~(ii), we have
$0\in C(0,U_k)$
for large $k$, and for such $k$
we call
$\partial U_k\cap C(0,U_k)$ the \emph{inner boundary}
of $U_k$ and denote it by $\partial_0 U_k$.
It is not difficult to see when $U_k$ is infinitely
connected, then the components of $\partial U_k$
must cluster at the inner or outer boundary (or both).
However, as we will explain below,
it can happen that they cluster at only one of them.

In the proof of Theorem~\ref{thm2} we have
shown that, with the terminology used there,
$X_k=D\left(0,\sqrt{a_{k+1}}\right)
\setminus C(0,U_k)\subset U_k$.
This implies that the components of
$\partial U_k$ do not cluster at the inner boundary
$\partial_0 U_k$:
it is an isolated part of the boundary in the sense that
$\dist(\partial_0 U_k,\partial U_k\setminus\partial_0 U_k)>0$.

On the other hand,
if $f$ is as in Theorem~\ref{thm4}, with $(P_k)$
is chosen such that
\begin{equation}\label{pr1}
\limsup_{k\to\infty}{kP_k} <\frac{|C|}{2e},
\end{equation}
then the outer boundary of $U_k$ is an isolated part of the
boundary. In fact, the proof of Theorem~\ref{thm4} shows
that~\eqref{pr1} implies that~\eqref{p4x} holds for all large~$k$.
We deduce that $S_{k+1}=\widetilde{U_{k+1}}\setminus
\overline{D(0,2^{k}r_{k+1})}$ contains no critical values. Thus the
components of $f^{-1}(S_{k+1})$ are doubly connected. Proceeding as
in the proof of~\eqref{p4e}, it is not difficult to see  that
$|f(z)|\geq 2^k r_{k+1}$ for $|z|=2^{k-1}r_k$. This implies that
there exists a component $T_k$ of $f^{-1}(S_{k+1})$  which
contains~$S_k$. It follows that $f(S_k)\subset f(T_k)=S_{k+1}$, and
this implies that $S_k\subset F(f)$. Thus the outer boundary
$\partial_\infty U_k$ of $U_k$ is isolated.

A similar reasoning shows that if
$$\liminf_{k\to\infty}{kP_k} >\frac{|C|}{2e},$$
then $\partial_0 U_k$ is isolated, and if
$$\liminf_{k\to\infty}{kP_k} <\frac{|C|}{2e}<\limsup_{k\to\infty}{kP_k},$$
then neither $\partial_0 U_k$ nor $\partial_\infty U_k$ is isolated.

These examples illustrate a result in~\cite{BRS} where
it is shown the inner or outer boundary are isolated
under a suitable hypothesis on the location of the
critical points.
\end{remark}

\section{Baker's example of an infinitely connected Fatou component}
\label{baker-ex}

In Baker's example~\eqref{1c}  the constant $C$ and the sequence
$(r_k)$ are chosen as follows. First let $0<C<1/(4e^2)$ and $r_1>1$.
Then choose $k_0$ such that $2^{k_0-1} C>2r_1$. Finally choose
$(r_k)$ such  that $r_{k+1}\geq 2r_k$ for $1\leq k< k_0$ and put
$$
r_{k+1}=C^2\prod_{j=1}^{k}\left(1+\frac{r_k}{r_j}\right)^2
$$
for $k\geq k_0$.

Baker showed that with
$$
s_k=\frac{k+1}{k+2}r_{k+1}
 \quad\text{and}\quad
B_k=\ann\left(r_k^2,s_k\right)
$$
 we have $f(B_k)\subset B_{k+1}$ for large $k$ and thus
$B_k\subset U_k$ for some multiply connected wandering
domain~$U_k$. In order to prove that $U_k$ is infinitely
connected, Baker proved that there exists a critical
point $c_{k+1}\in(-s_k,-r_k^2)\subset B_k$ for large~$k$.
The infinite connectivity then follows from Lemma~\ref{lemma1a}.

Analogously to~\eqref{p4i} we now find that
$$
c_k=\left(1-\frac{1}{k+\delta_k}\right)a_k,
$$
where $\delta_k\to 0$, and instead of~\eqref{p4j} we now obtain
$$
|f(c_k)|\sim \frac{C^2}{ek}
\prod_{j=1}^{k-1}\left(1+\frac{r_k}{r_j}\right)^2
=\frac{1}{4ek}
r_{k+1}.
$$
Similarly as in~\eqref{p4a} we also have
$r_{k+1}\geq 2^k r_k^2$.
It follows that $|f(c_k)|/r_k^2\to\infty$.
As $c_{k+1}$ is the only critical point in $U_k$,
it now follows from Theorem~\ref{thm1}, applied with
$A_{k+1}=\ann(r_{k+1}^2,|f(c_{k+1})|)\subset U_{k+1}$,
that $\partial U_k$ is not uniformly perfect.

\end{document}